\documentclass[english,12pt]{article}
\usepackage{amssymb}
\usepackage{amsmath}
\usepackage{amsfonts}
\usepackage{mathrsfs}
\usepackage{graphicx}
\usepackage{tikz}
\usepackage{calc}
\tikzstyle{vertex}=[circle, draw, inner sep=0pt, minimum size=6pt]
\newcommand{\vertex}{\node[vertex]}
\usetikzlibrary{decorations.markings}
\usetikzlibrary{arrows,shapes,positioning}
\tikzstyle arrowstyle=[scale=1.5]%此命令是控制箭头的大小。
\tikzstyle directed=[postaction={decorate,decoration={markings, mark=at position .75 with {\arrow[arrowstyle]{stealth}}}}] %此命令是控制正向箭头的位置。
\tikzstyle reverse directed=[postaction={decorate,decoration={markings, mark=at position .75 with {\arrowreversed[arrowstyle]{stealth};}}}]%此命令是控制反向箭头的位置。
\newtheorem{Theorem}{Theorem}[section]

\newtheorem{Proposition}{Proposition}[section]

\newenvironment{proof}{\noindent {\bf Proof.}\rm}
{\mbox{}\hfill\rule{0.5em}{0.809em}\par}

\usepackage{algorithm,algcompatible,amsmath}
\algnewcommand\INPUT{\item[\textbf{Input:}]}%
\algnewcommand\OUTPUT{\item[\textbf{Output:}]}%
\begin{document}
\title{ \huge Antidirected hamiltonian paths in $k$-hypertournaments\thanks{This work was supported by  NSFC (No. 12261016)}}
\author{
Hong Yang\thanks{College of Mathematics and System Sciences, Xinjiang University, Urumqi,
China. Email: honger9506@126.com},
Changchang Dong\thanks{Department of Mathematical Sciences, Tsinghua University, Beijing 100084, China. Email: jintiandcc@126.com},
Jixiang Meng\thanks{College of Mathematics and System Sciences, Xinjiang University, Urumqi,
China. Email: mjx@xju.edu.cn},
Juan Liu\thanks{College of Big Data Statistics, Guizhou University of Finance and Economics, Guiyang, 550025, China. Email: liujuan1999@126.com(Corresponding author)}
 }
\date{}

\maketitle

\begin{abstract}
A $k$-hypertournament $H$ on $n$ vertices is a pair $(V(H),A(H))$, where $V(H)$ is a set of vertices and $A(H)$ is a set of $k$-tuples of vertices, called arcs, such that for any $k$-subset $S$ of $V(H)$, $A(H)$ contains exactly one of the $k!$ $k$-tuples whose
 entries belong to $S$. Clearly, a 2-hypertournament is a tournament.
 An antidirected path in $H$ is a sequence $x_1 a_1 x_2 a_2 x_3 \ldots x_{t-1} a_{t-1} x_t$ of distinct vertices $x_1, x_2, \ldots, x_t$ and distinct arcs $a_1, a_{2},\ldots, a_{t-1}$ such that
  for any $i\in \{2,3,\ldots, t-1\}$, either $x_{i-1}$ precedes $x_{i}$ in $a_{i-1}$ and $x_{i+1}$ precedes $x_{i}$ in $a_{i}$, or $x_{i}$ precedes $x_{i-1}$ in $a_{i-1}$ and $x_{i}$ precedes $x_{i+1}$ in $a_{i}$.
An antidirected path that includes all vertices of $H$ is known as an antidirected hamiltonian path.
In  this paper, we prove that except for four hypertournaments, $T_3^{c}, T_5^{c}, T_7^{c}$ and $H_{4}$,
 every $k$-hypertournament with $n$ vetices, where $2\leq k\leq n-1$, has an antidirected hamiltonian path, which extends Gr\"{u}nbaum's theorem on tournaments (except for three tournaments, $T_3^{c}, T_5^{c}$ and $T_7^{c}$,
 every tournament has an antidirected hamiltonian path).
\end{abstract}


\section{Introduction}
~~~ Let $D$ be a digraph  with vertex set $V(D)$ and arc set $A(D)$. In this paper, the notation $(x,y)$ is used to represent an arc in the digraph $D$ that is oriented from vertex $x$ to vertex $y$. Additionally, the notation $x\rightarrow y$ may also be used to denote the arc.
A digraph is \textbf{tournament} if it has no pair of nonadjacent vertices and cycles of length 2. We denote the tournament obtained from a tournament $T$, by reversing the direction
of every arc in it by $\overleftarrow{T}_{n}$.
 Undefined terms and notation follow \cite{BaGu09}.
 \par An \textbf{antidirected path} of digraph $D$ is a path that alternates between forward and backward arcs.
An \textbf{antidirected hamiltonian path} of $D$ is an antidirected path that contains every vertex in $H$.
 An antidirected hamiltonian path in $D$ of the form $x \rightarrow x_1 \leftarrow \cdots y$ (resp. $x \leftarrow x_1 \rightarrow \cdots y$) is said to have $x$ as a \textbf{starting vertex} (resp. \textbf{terminal vertex}). If $x$ is both the starting and terminating vertex, then it is known as a \textbf{double point}.
 The study of antidirected hamiltonian paths in tournaments was initiated by Gr\"{u}nbaum \cite{Grnbaum1971}, Gr\"{u}nbaum introduced antidirected hamiltonian paths and proved that every tournament has an antidirected hamiltonian path except for the three tournaments $T_3^{c}, T_5^{c}$, and $T_7^{c}$ which are depicted as in Figure 1.
In 1972, Rosenfeld \cite{Rosenfeld1972} proved that if tournament $T_n$ has an antidirected hamiltonian path and $n$ is odd, then $T_n$ has a double point.
  Additional researches on antidirected hamiltonian path in digraphs can be found in \cite{Chen2024,Hell2002,Klimo2022,Sahili2018,Thomassen1973}, among others.
  {\tikzstyle directed=[postaction={decorate,decoration={markings, mark=at position .52 with {\arrow[arrowstyle]{stealth}}}}]
\tikzstyle reverse directed=[postaction={decorate,decoration={markings, mark=at position .52 with {\arrowreversed[arrowstyle]{stealth};}}}]
\[\begin{tikzpicture}
[x=0.85cm, y=0.65cm]
%%%%%%%%%%%%%%%%%%%%%%%%%%%%%%%%%%%%%%%%%%%%%%%%%%%%%%%%%%%%%%%%%%%%%%%%%%%%%%%%%
\vertex  (x1) at (-7,2)[fill=black] [label=above:$z_{1}$] {};
\vertex  (x3) at (-8.6,-2)[fill=black] [label=left:$z_{2}$] {};
\vertex  (x2) at (-5.4,-2)[fill=black] [label=right:$z_{3}$] {};
%左图布点
%%%%%%%%%%%%%%%%%%%%%%%%%%%%%%%%%%%%%%%%%%%%%%%%%%%%%%%%%%%%%%%%%%%%%%%%%%%%%%%%%%%%%%%%%
\draw[thick,directed](x1)--(x2);
\draw[thick,directed](x2)--(x3);
\draw[thick,directed](x3)--(x1);
%左图弧线
%%%%%%%%%%%%%%%%%%%%%%%%%%%%%%%%%%%%%%%%%%%%%%%%%%%%%%%%%%%%%%%%%%%%%%%%%%%%%%%%%%%%%%%%%
%The above codes for left digrah
%%%%%%%%%%%%%%%%%%%%%%%%%%%%%%%%%%%%%%%%%%%%%%%%%%%%%%%%%%%%%%%%%%%%%%%%%%%%%%%%%%%%%%%%%
\vertex  (x1) at (-1,3)[fill=black] [label=above:$z_{1}$] {};
\vertex  (x2) at (1,1.1)[fill=black] [label=right:$z_{2}$] {};
\vertex  (x3) at (0.3,-2)[fill=black] [label=right:$z_{3}$] {};
\vertex  (x4) at (-2.3,-2)[fill=black] [label=left:$z_{4}$] {};
\vertex  (x5) at (-3,1.1)[fill=black] [label=left:$z_{5}$] {};
%中图布点
%%%%%%%%%%%%%%%%%%%%%%%%%%%%%%%%%%%%%%%%%%%%%%%%%%%%%%%%%%%%%%%%%%%%%%%%%%%%%%%%%%%%%%%%%
\draw[thick,directed](x1)--(x2);\draw[thick,directed](x2)--(x3);
\draw[thick,directed](x3)--(x4);\draw[thick,directed](x4)--(x5);
\draw[thick,directed](x5)--(x1);
\draw[directed](x1)--(x3);\draw[directed](x3)--(x5);
\draw[directed](x5)--(x2);\draw[directed](x2)--(x4);
\draw[directed](x4)--(x1);
%中间图弧线
%%%%%%%%%%%%%%%%%%%%%%%%%%%%%%%%%%%%%%%%%%%%%%%%%%%%%%%%%%%%%%%%%%%%%%%%%%%%%%%%%%%%%%%%%
%The above codes for middle digrah
%%%%%%%%%%%%%%%%%%%%%%%%%%%%%%%%%%%%%%%%%%%%%%%%%%%%%%%%%%%%%%%%%%%%%%%%%%%%%%%%%%%%%%%%%
\vertex  (x1) at (6,3.3)[fill=black] [label=above:$z_{1}$] {};
\vertex  (x2) at (7.8,2.2)[fill=black] [label=right:$z_{2}$] {};
\vertex  (x3) at (8.3,-0.2)[fill=black] [label=right:$z_{3}$] {};
\vertex  (x4) at (7.1,-2.1)[fill=black] [label=right:$z_{4}$] {};
\vertex  (x5) at (4.9,-2.1)[fill=black] [label=left:$z_{5}$] {};
\vertex  (x6) at (3.7,-0.2)[fill=black] [label=left:$z_{6}$] {};
\vertex  (x7) at (4.2,2.2)[fill=black] [label=left:$z_{7}$] {};
%中图布点
%%%%%%%%%%%%%%%%%%%%%%%%%%%%%%%%%%%%%%%%%%%%%%%%%%%%%%%%%%%%%%%%%%%%%%%%%%%%%%%%%%%%%%%%%
\draw[thick,directed](x1)--(x2);\draw[thick,directed](x2)--(x3);
\draw[thick,directed](x3)--(x4);\draw[thick,directed](x4)--(x5);
\draw[thick,directed](x5)--(x6);\draw[thick,directed](x6)--(x7);
\draw[thick,directed](x7)--(x1);
\draw[directed](x1)--(x3);\draw[directed](x3)--(x5);
\draw[directed](x5)--(x7);\draw[directed](x7)--(x2);
\draw[directed](x2)--(x4);\draw[directed](x4)--(x6);
\draw[directed](x6)--(x1);
\draw[directed](x1)--(x5);\draw[directed](x5)--(x2);
\draw[directed](x2)--(x6);\draw[directed](x6)--(x3);
\draw[directed](x3)--(x7);\draw[directed](x7)--(x4);
\draw[directed](x4)--(x1);
%The above codes for right digrah
%%%%%%%%%%%%%%%%%%%%%%%%%%%%%%%%%%%%%%%%%%%%%%%%%%%%%%%%%%%%%%%%%%%%%%%%%%%%%%%%%%%%%%%%%
\node at (-7,-2.8){$T^{c}_{3}$};\node at (-1,-2.8){$T^{c}_{5}$};\node at (6,-2.8){$T^{c}_{7}$};
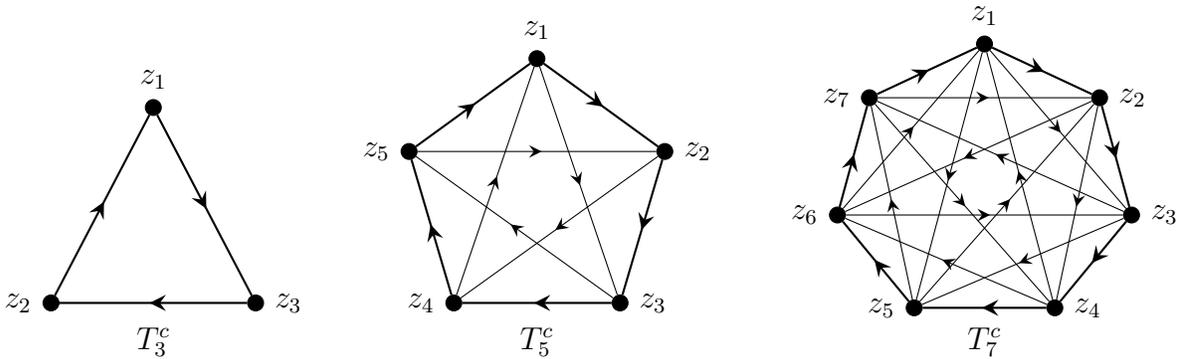
\node at (-1,-3.8){Figure 1. Tournaments $T^{c}_{3}$, $T^{c}_{5}$ and $T^{c}_{7}$.};
\end{tikzpicture}\]
\begin{Theorem}\label{ADHT}(\cite{Grnbaum1971})(Gr\"{u}nbaum's theorem)
Except for three tournaments, $T_3^{c}, T_5^{c}$ and $T_7^{c}$, every tournament has an antidirected hamiltonian path.
\end{Theorem}
%\begin{Theorem}\label{ADHO} (\cite{Rosenfeld1972})
%If tournament $T_n$ has an antidirected hamiltonian path and $n$ is odd, then $T_n$ has a double point.
%\end{Theorem}
\par A \textbf{$k$-hypertournament} $H$ on $n$ vertices is a pair $(V(H),A(H))$, where $V(H)$ is a set of vertices and $A(H)$ is a set of $k$-tuples of vertices, called \textbf{arcs}, such that for any $k$-subset $S$ of $V(H)$, $A(H)$ contains exactly one of the $k!$ $k$-tuples whose
 entries belong to $S$.  Clearly, a 2-hypertournament is
merely a tournament. We use $(x_{1},x_{2},\ldots,x_{k})$ to denote the orientation of arc $a$ in $H$,
 where $\{x_{1},x_{2},\ldots,x_{k}\}\subseteq V(H)$, for any two distinct integers $i$ and $j$ with $1\leq i<j\leq k$, we say $x_{i}$ precedes $x_{j}$ in $a$.
\par Given a $k$-hypertournament $H=(V(H), A(H))$, for two distinct vertices $x, y \in V(H)$, the set of all arcs in which $x$ precedes $y$ is denoted by $A_H(x, y)$.
 Clearly,
$$
|A_H(x, y)|+|A_H(y, x)|=\Big(\begin{array}{l}
n-2 \\
k-2
\end{array}\Big).
$$
%For a vertex $x \in V(H)$, the \textbf{out-degree} of $x$, denoted by $d_{H}^{+}(x)$, and the \textbf{in-degree} of $x$, denoted by $d_{H}^{-}(x)$, are defined as $d_{H}^{+}(x)=\sum_{y \in V(H) \backslash\{x\}}|A_H(x, y)|$ and $d_{H}^{-}(x)=\sum_{y \in V(H) \backslash\{x\}}|A_H$$(y,x)|$.
\par Let $H$ be a $k$-hypertournament. For an integer $t$ with at least 2, an \textbf{antidirected path} $P$ with vertex set $V(P)$  and arc set $A(P)$ in $H$
is a sequence $x_1 a_1 x_2 a_2 x_3 \ldots x_{t-1} a_{t-1} x_t$ of distinct vertices $x_1, x_2, \ldots, x_t$ and distinct arcs $a_1, a_{2},\ldots, a_{t-1}$ such that for any $i\in \{2,3,\ldots, t-1\}$, either $x_{i-1}$ precedes $x_{i}$ in $a_{i-1}$ and $x_{i+1}$ precedes $x_{i}$ in $a_{i}$, or $x_{i}$ precedes $x_{i-1}$ in $a_{i-1}$ and $x_{i}$ precedes $x_{i+1}$ in $a_{i}$.
Sometimes, we use $x_{1}$ $^{\underrightarrow{a_{1}}}x_{2}$ $^{\underleftarrow{a_{2}}}x_{3}\cdots x_{t}$ (or $x_{1}$ $^{\underleftarrow{a_{1}}}x_{2}$ $^{\underrightarrow{a_{2}}}x_{3}\cdots x_{t}$) to denote the antidirected path $P$.
An \textbf{antidirected hamiltonian path} of $H$ is an antidirected path that contains every vertex in $H$.
\par The tournaments, as a special type of $k$-hypertournament, there exist two of the most basic theorems about hamiltonian paths and hamiltonian cycles: every tournament has a hamiltonian path (R\'{e}dei's theorem) \cite{Redei1934}, and every strong tournament
has a hamiltonian cycle (Camion's theorem) \cite{Camion1959}. It is natural for one to ask whether these theorems hold for hypertournaments as well.
Gutin \cite{Gutin1997} proved that every $k$-hypertournament on $n\geq k+1$ vertices has a hamiltonian path, and every strong $k$-hypertournament on $n\geq k+2$ vertices has a hamiltonian cycle.
\par Furthermore, Gutin asked whether other well-known results on tournaments could be extended to hypertournaments.  Petrovi\'{c} \cite{Petrovic2006}, Yang \cite{Yang2009} and Li \cite{Li2013} extend Moon's theorem \cite{Moon1966} on tournaments (every strong tournament is vertex-pancyclic) to hypertournaments. Particularly, Li \cite{Li2013} proved that every strong $k$-hypertournament with $n$ vertices, where $3\leq k \leq
n-2$, is vertex-pancyclic, making this the best possible result. The results of arc-pancyclic and pancyclic arc in tournaments have also been extended to hypertournaments by a number of authors, who have obtained some related results (cf.  Surmacs \cite{Surmacs2017}, Xiao \cite{Lan2023}, Guo \cite{Guo2014} and Li \cite{Li2016}).
\par In this paper, we want to turn our attention to the antidirected hamiltonian paths of hypertournaments. Our goal is to extend Gr\"{u}nbaum's theorem to hypertournaments. Furthermore, we prove that except for four hypertournaments, $T_3^{c}, T_5^{c}, T_7^{c}$ and $H_{4}$,
 every $k$-hypertournament with $n$ vetices, where $2\leq k\leq n-1$, has an antidirected hamiltonian path.
\section{Antidirected hamiltonian paths in $k$-hypertournaments}
~~~ In this section, we study the existence of antidirected hamiltonian path of $k$-hypertournament with $n$ vertices, where $2\leq k\leq n-1$. Firstly, we define a $3$-hypertournament with $4$ vertices, denoted by $H_{4}$, such that
 $V(H_{4})=\{1,2,3,4\}$ and $A(H_{4})=\{(2,3,4),(4,1,2),(3,4,1),(1,2,3)\}$.
 Next, we prove that $H_{4}$ does not contain antidirected hamiltonian path.
 \begin{Proposition}\label{NADH}
Let $H$ be a $k$-hypertournament with $n$ vertices. If $H\cong H_{4}$, then $H$ does not contain an antidirected hamiltonian path.
\end{Proposition}
 \begin{proof}
 To prove that, suppose that $H$ contains an antidirected hamiltonian path $P$, then by the definition of antidirected hamiltonian path, we may assume that $P=x_{1}$ $^{\underrightarrow{a_{1}}}x_{2}$ $^{\underleftarrow{a_{2}}}x_{3}$ $^{\underrightarrow{a_{3}}}x_{4}$, where $\{x_{1},x_{2},x_{3},x_{4}\}=\{1,2,3,4\}$.
 Assume that $x_{2}=1$, then $\{x_{1},x_{3}\}=\{3,4\}, x_{4}=2$ and $\{a_{1},a_{2}\}=\{(4,1,2),$
 $(3,4,1)\}$. Then $a_{3}\in A_{H}(3,2)\cup A_{H}(4,2)$,  but, this is impossible because $A(H_{4})-\{(4,1,2),$
 $(3,4,1)\}=\{(2,3,4),(1,2,3)\}$ and $\{(2,3,4),(1,2,3)\}\cap (A_{H}(3,2)\cup A_{H}(4,2))=\emptyset$. Likewise, if $x_{2}\in \{2,3,4\}$, then a contradiction occurs. Therefore, $H$ does not contain an antidirected hamiltonian path. This completes the proof of the proposition.
 \end{proof}

The main result of this section is the following.
\begin{Theorem}\label{ADH}
Let $H$ be a $k$-hypertournament with $n$ vertices, where $2\leq k\leq n-1$. Then $H$ contains an antidirected hamiltonian path if and only if $H\not\in \{ H_{4},T^{c}_{3},T^{c}_{5},T^{c}_{7}\}$.
\end{Theorem}
\begin{proof}
If $H\in \{T^{c}_{3},T^{c}_{5},T^{c}_{7}\}$, then by Theorem \ref{ADHT},
$H$ does not contain an antidirected hamiltonian path.
If $H\cong H_{4}$, then by Proposition \ref{NADH}, $H$ does not contain an antidirected hamiltonian path. Hence suppose that $H\not\in \{ H_{4},T^{c}_{3},T^{c}_{5},T^{c}_{7}\}$, we want to prove that
$H$ contains an antidirected hamiltonian path.
 \par Let $H=(V(H), A(H))$ be a $k$-hypertournament with $V(H)=\{1,2, \ldots, n\}$. We consider the cases $k=n-1$ and $k<n-1$ separately.\\
 \textbf{Case 1} $k=n-1$.
 \par Then $|A(H)|=\big(\begin{array}{l}
n \\
k
\end{array}\big)=n$. We proceed by induction on $k \geq 2$. By Theorem \ref{ADHT}, this theorem holds for $k=2$. Hence, suppose that $k \geq 3$. Without loss of generality, assume that $H$ contains the arc $a$ that has the vertices $2,3, \ldots, n$ and for any $i\in \{2,3,\ldots,n-1\}$, vertex $i$ precedes vertex $n$ in arc $a$.
  Let $b$ be the $\operatorname{arc}$ of $H$ that has the vertices $1,2, \ldots, n-1$. Consider the $(k-1)$-hypertournament $H'=(V(H'), A(H'))$ obtained from $H$ by deleting the arc $a$, deleting vertex $n$ from the arcs in $A(H)-\{a, b\}$, and finally deleting vertex 1 from $b$. So, $V(H')=\{1,2, \ldots, n-1\}$, $A(H')=\{e': e'$ is $e$ without vertex $n, e \in A(H)-\{a, b\}\} \cup\{b'\}$, where $b'$ is $b$ without the vertex 1. As $n=k+1\geq 4$, we have $H'\not\in \{T^{c}_{5},T^{c}_{7}\}$.
 \par Assume first that $H'\cong T^{c}_{3}$. Then without loss of generality, assume that
\[\mbox{
 $A(H')=\{a'_{1}=(1,2), a'_{2}=(3,1),b'=(2,3)\}$.
 }\]
 These arcs correspond to arcs $a_{1},a_{2},b$ in $A(H)$ such that $a_{1}$ contains vertices 1,2,4 and $a_{2}$ contains vertices 1,3,4.
If $a_{1}\in A_{H}(1,4)$, then $1$ $^{\underrightarrow{a_{1}}}4$ $^{\underleftarrow{a}}2$ $^{\underrightarrow{b}}3$
 is an antidirected hamiltonian path in $H$. If $a_{1}=(4,1,2)$ and $a\in A_{H}(3,2)$, then $2$ $^{\underleftarrow{a}}3$ $^{\underrightarrow{a_{2}}}1$ $^{\underleftarrow{a_{1}}}4$
 is an antidirected hamiltonian path in $H$.
 Hence assume that
\begin{equation}\label{a1}
  \mbox{$a_{1}=(4,1,2)$ and $a=(2,3,4)$.}
\end{equation}
If $b\in A_{H}(2,1)$, then $3$ $^{\underleftarrow{a}}2$ $^{\underrightarrow{b}}1$ $^{\underleftarrow{a_{1}}}4$
 is an antidirected hamiltonian path in $H$. If $a_{2}\in A_{H}(1,4)$, then
 $2$ $^{\underleftarrow{a_{1}}}1$ $^{\underrightarrow{a_{2}}}4$ $^{\underleftarrow{a}}3$
 is an antidirected hamiltonian path in $H$. Hence assume that
\begin{equation}\label{b}
  \mbox{$b=(1,2,3)$ and $a_{2}\in A_{H}(4,1)$.}
\end{equation}
We can claim $a_{2}\in A_{H}(4,3)$. Otherwise, if $a_{2}\in A_{H}(3,4)$, then
$a_{2}=(3,4,1)$.
This, together with (\ref{a1}) and (\ref{b}), imply that $H\cong H_{4}$, contrary to $H\not\in \{ H_{4},T^{c}_{3},T^{c}_{5},T^{c}_{7}\}$. Hence $a_{2}\in A_{H}(4,3)$, and so
 $4$ $^{\underrightarrow{a_{2}}}3$ $^{\underleftarrow{b}}1$ $^{\underrightarrow{a_{1}}}2$
 is an antidirected hamiltonian path in $H$.
\par Assume now that $H'\cong H_{4}$, then suppose that  $A(H')=\{b'=(2,3,4),a'_{1}=(4,1,2),a'_{2}=(3,4,1),a'_{3}=(1,2,3)\}$. These arcs correspond to arcs $b,a_{1},a_{2},a_{3}$ in $A(H)$ such that $a_{1}$ contains vertices 1,2,4,5,  $a_{2}$ contains vertices 1,3,4,5 and $a_{3}$ contains vertices 1,2,3,5. If $a_{1}\in A_{H}(5,4)$ and $a\in A_{H}(2,3)$, then $1$ $^{\underrightarrow{a_{3}}}3$ $^{\underleftarrow{a}}2$ $^{\underrightarrow{b}}4$$^{\underleftarrow{a_{1}}}5$
 is an antidirected hamiltonian path in $H$;
 if $a_{1}\in A_{H}(5,4)$ and $a\in A_{H}(3,2)$, then $5$ $^{\underrightarrow{a_{1}}}4$ $^{\underleftarrow{a_{2}}}3$ $^{\underrightarrow{a}}2$$^{\underleftarrow{a_{3}}}1$
 is an antidirected hamiltonian path in $H$;
  if $a_{1}\in A_{H}(4,5)$, then
 $1$ $^{\underleftarrow{a_{2}}}4$ $^{\underrightarrow{a_{1}}}5$ $^{\underleftarrow{a}}2$ $^{\underrightarrow{b}}3$
 is an antidirected hamiltonian path in $H$.
\par Finally, assume that $H'\not\in \{T_{3}^{c},H_{4}\}$, and so $H'\not\in \{ H_{4},T^{c}_{3},T^{c}_{5},T^{c}_{7}\}$.
By the induction hypothesis, $H'$ has an antidirected hamiltonian path.
If $n$ is odd, then $n-1$ is even and $H'$ contains an antidirected hamiltonian path with one ending vertex as starting vertex and another ending vertex as terminal vertex; if $n$ is even, then $n-1$ is odd and, $H'$ contains an antidirected hamiltonian path such that two ending vertices are starting vertices or $H'$ contains an antidirected hamiltonian path such that two ending vertices are terminal vertices.
\par Assume first that $y_{1}$ $^{\underleftarrow{b'_{1}}}y_{2}$ $^{\underrightarrow{b'_{2}}}y_{3}\cdots y_{n-2}$$^{\underrightarrow{b'_{n-2}}}y_{n-1}$ is an antidirected hamiltonian path of $H'$ such that $y_{1}$ and $y_{n-1}$ are terminal vertices. This antidirected hamiltonian path corresponds to the antidirected path $P=y_{1}$ $^{\underleftarrow{b_{1}}}y_{2}$ $^{\underrightarrow{b_{2}}}y_{3}\cdots y_{n-2}$$^{\underrightarrow{b_{n-2}}}y_{n-1}$ in $H$. Clearly, $\{y_1, y_{2} \ldots, y_{n-1}\}=$ $\{1,2, \ldots, n-1\}$ and $A(H)-\{b_1, \ldots, b_{n-2}\}$ consists of the arc $ a$ and another arc $d$.
If $y_{1}\not=1$ and $y_{n-1}\not=1$, then $a$ contains vertices $y_{1}$ and $y_{n-1}$, and so $a\in A_{H}(y_{1},y_{n-1})$ or $a\in A_{H}(y_{n-1},y_{1})$. Thus, $y_{2}$ $^{\underrightarrow{b_{2}}}y_{3}\cdots y_{n-2}$$^{\underrightarrow{b_{n-2}}}y_{n-1}$ $^{\underleftarrow{a}}y_{1}$ or $y_{n-1}$ $^{\underrightarrow{a}}y_{1}$ $^{\underleftarrow{b_{1}}}y_{2}$ $^{\underrightarrow{b_{2}}}y_{3}\cdots y_{n-2}$ is an antidirected hamiltonian of $H'$. Likewise, if $d$ contains vertices $y_{1}$ and $y_{n-1}$, then an antidirected hamiltonian of $H'$ can be found. Hence assume that $y_{1}=1$ or $y_{n-1}=1$, and $d$ does not contain vertex $y_{1}$ or $d$ does not contain vertex $y_{n-1}$. Then $y_{2}\not=1$ and $y_{n-2}\not=1$. Since $k=n-1$, we have that
$d$ only does not contain vertex $y_{1}$ or $d$ only does not contain vertex $y_{n-1}$. Without loss of generality, assume that $d$ only does not contain vertex $y_{1}$. Then $d$ contains vertices $n$ and $y_{n-1}$. If $d\in A_{H}(n,y_{n-1})$, then $y_{1}$ $^{\underleftarrow{b_{1}}}y_{2}$ $^{\underrightarrow{b_{2}}}y_{3}\cdots y_{n-2}$$^{\underrightarrow{b_{n-2}}}y_{n-1}$$^{\underleftarrow{d}} n$
is an antidirected hamiltonian path in $H$; if $d\in A_{H}(y_{n-1},n)$, then $y_{1}$ $^{\underleftarrow{b_{1}}}y_{2}$ $^{\underrightarrow{b_{2}}}y_{3}\cdots y_{n-2}$ $^{\underrightarrow{a}}n$$^{\underleftarrow{d}}x_{n-1}$
is an antidirected hamiltonian path in $H$.

\par Assume now that
$x_{1}$ $^{\underrightarrow{a'_{1}}}x_{2}$ $^{\underleftarrow{a'_{2}}}x_{3}\cdots x_{n-1}$ is an antidirected hamiltonian path of $H'$ such that $x_{1}$ is starting vertex. This antidirected hamiltonian path corresponds to the antidirected path $Q=x_{1}$ $^{\underrightarrow{a_{1}}}x_{2}$ $^{\underleftarrow{a_{2}}}x_{3}\cdots x_{n-1}$ in $H$. Clearly, $\{x_1, x_{2} \ldots, x_{n-1}\}=$ $\{1,2, \ldots, n-1\}$ and $A(H)-\{a_1, \ldots, a_{n-2}\}$ consists of the arc $ a$ and another arc $c$. It follows that
\[
\mbox{arcs $a_{1},a_{2},\ldots,a_{n-2}$ and $c$ all include the vertex $1$.}
\]

\par If $x_{1}\not=1$, then $n$$^{\underleftarrow{a}}x_{1}$ $^{\underrightarrow{a_{1}}}x_{2}$ $^{\underleftarrow{a_{2}}}x_{3}\cdots x_{n-1}$ is an antidirected hamiltonian path in $H$.
If $x_{1}=1$ and $n$ is even, then $a_{n-2}\in A_{H}(x_{n-1},x_{n-2})$ and $x_{1}$ $^{\underrightarrow{a_{1}}}x_{2}$ $^{\underleftarrow{a_{2}}}x_{3}\cdots x_{n-2}$ $^{\underleftarrow{a_{n-2}}}x_{n-1}$$^{\underrightarrow{a}} n$ is an antidirected hamiltonian path in $H$. Hence from now on assume that
\begin{equation*}
  \mbox{$x_{1}=1$ and $n=k+1\geq 5$ is odd.}
\end{equation*}
And so
 \begin{equation*}
  \mbox{$a_{n-2}\in A_{H}(x_{n-2},x_{n-1})$.}
\end{equation*}
Next,  consider two subcases.\\
 \textbf{Subcase 1.1} $c\not=b$.
\par Then the arc $c$ contains vertices 1 and $n$ as $a\not=c$ and $k=n-1$.
Suppose that the arc $c$ contains the vertex $x_{n-1}$.
If $c\in A_{H}(n,x_{n-1})$, then $1$ $^{\underrightarrow{a_{1}}}x_{2}$ $^{\underleftarrow{a_{2}}}x_{3}\cdots x_{n-2}$ $^{\underrightarrow{a_{n-2}}}x_{n-1}$$^{\underleftarrow{c}} n$
is an antidirected hamiltonian path in $H$; if $c\in A_{H}(x_{n-1},n)$, then $1$ $^{\underrightarrow{a_{1}}}x_{2}$ $^{\underleftarrow{a_{2}}}x_{3}\cdots x_{n-3}$$^{\underleftarrow{a_{n-3}}}x_{n-2}$ $^{\underrightarrow{a}}n$$^{\underleftarrow{c}}x_{n-1}$
is an antidirected hamiltonian path in $H$. Hence assume that the arc $c$ does not contain the vertex $x_{n-1}$. And hence
\begin{equation*}
  \mbox{arcs $a_{1},a_{2},\ldots,a_{n-2}$ all contain vertex $x_{n-1}$ and arc $c$ contains vertices $1,x_{2},\ldots,x_{n-2},n$.}
\end{equation*}
If $c\in  A_{H}(1,n)$, then $n$$^{\underleftarrow{c}}1$ $^{\underrightarrow{a_{1}}}x_{2}$ $^{\underleftarrow{a_{2}}}x_{3}\cdots x_{n-1}$ is an antidirected hamiltonian path in $H$; if $a_{n-2}\in A_{H}(1,x_{n-1})$, then $x_{n-1}$ $^{\underleftarrow{a_{n-2}}}1$ $^{\underrightarrow{a_{1}}}x_{2}$ $^{\underleftarrow{a_{2}}}x_{3}\cdots x_{n-2}$$^{\underrightarrow{a}}n$
is an antidirected hamiltonian path in $H$. Hence assume that
 $c\in  A_{H}(n,1)$ and $a_{n-2}\in A_{H}(x_{n-1},1)$.
 \par Suppose that the arc $a_{1}$ contains the vertex $n$.
 If $a_{1}\in A_{H}(n,x_{n-1})$, then $x_{2}$ $^{\underleftarrow{a_{2}}}x_{3}$ $^{\underrightarrow{a_{3}}}x_{4}\cdots x_{n-2}$
 $^{\underrightarrow{a_{n-2}}}$
 $ x_{n-1}$ $^{\underleftarrow{a_{1}}}n$$^{\underrightarrow{c}}1$
 is an antidirected hamiltonian path in $H$. If $a_{1}\in A_{H}(x_{n-1},n)$, then
$x_{2}$ $^{\underleftarrow{a_{2}}}x_{3}$ $^{\underrightarrow{a_{3}}}x_{4}\cdots x_{n-2}$$^{\underrightarrow{a}}n$ $^{\underleftarrow{a_{1}}}x_{n-1}$$^{\underrightarrow{a_{n-2}}}1$
is an antidirected hamiltonian path in $H$.
\par Assume now that the arc $a_{1}$ does not contain the vertex $n$, then $a_{1}=b$. Thus,
\begin{equation*}
  \mbox{arcs $a_{2},a_{3},\ldots,a_{n-2}$ all contain vertex $n$.}
\end{equation*}
If $a_{n-2}\in A_{H}(x_{n-2},n)$, then $1$ $^{\underrightarrow{a_{1}}}x_{2}$ $^{\underleftarrow{a_{2}}}x_{3}\cdots x_{n-3}$$^{\underleftarrow{a_{n-3}}}x_{n-2}$ $^{\underrightarrow{a_{n-2}}}n$$^{\underleftarrow{a}}x_{n-1}$
 is an antidirected hamiltonian path in $H$. If $a_{n-2}\in A_{H}(x_{n-1},n)$, then $1$ $^{\underrightarrow{a_{1}}}x_{2}$ $^{\underleftarrow{a_{2}}}x_{3}\cdots x_{n-3}$$^{\underleftarrow{a_{n-3}}}x_{n-2}$ $^{\underrightarrow{a}}n$$^{\underleftarrow{a_{n-2}}}x_{n-1}$
  is an antidirected hamiltonian path in $H$. Hence assume that
  \begin{equation*}
  \mbox{$a_{n-2}\in A_{H}(n,x_{n-2})\cap A_{H}(n,x_{n-1})$.}
\end{equation*}
  If $a\in A_{H}(x_{n-2},x_{n-1})$, then  $1$ $^{\underrightarrow{a_{1}}}x_{2}$ $^{\underleftarrow{a_{2}}}x_{3}\cdots x_{n-3}$$^{\underleftarrow{a_{n-3}}}x_{n-2}$ $^{\underrightarrow{a}}x_{n-1}$$^{\underleftarrow{a_{n-2}}}n$
  is an antidirected hamiltonian path in $H$. Hence assume that
\begin{equation*}
  \mbox{$a\in  A_{H}(x_{n-1},x_{n-2})$.}
\end{equation*}
If $c\in A_{H}(x_{n-2},n)$, then
$1$ $^{\underrightarrow{a_{1}}}x_{2}$ $^{\underleftarrow{a_{2}}}x_{3}\cdots x_{n-3}$$^{\underleftarrow{a_{n-3}}}x_{n-2}$ $^{\underrightarrow{c}}n$$^{\underleftarrow{a}}x_{n-1}$
 is an antidirected hamiltonian path in $H$; if $c\in A_{H}(n,x_{n-3})$, then
 $1$ $^{\underrightarrow{a_{1}}}x_{2}$ $^{\underleftarrow{a_{2}}}x_{3}\cdots x_{n-3}$$^{\underleftarrow{c}}n$ $^{\underrightarrow{a_{n-2}}}x_{n-2}$$^{\underleftarrow{a}}x_{n-1}$
 is an antidirected hamiltonian path in $H$. Hence assume that
\begin{equation*}
  \mbox{$c\in A_{H}(n,x_{n-2})\cap A_{H}(x_{n-3},n)$.}
\end{equation*}
\par If $a_{n-3}\in A_{H}(x_{n-1},x_{n-3})$, then
$1$ $^{\underrightarrow{a_{1}}}x_{2}$ $^{\underleftarrow{a_{2}}}x_{3}\cdots x_{n-3}$$^{\underleftarrow{a_{n-3}}}x_{n-1}$ $^{\underrightarrow{a}}x_{n-2}$$^{\underleftarrow{c}}n$
 is an antidirected hamiltonian path in $H$; if $a_{n-3}\in A_{H}(x_{n-3},x_{n-1})$ and $n\geq 7$, then
 $1$ $^{\underrightarrow{a_{1}}}x_{2}$ $^{\underleftarrow{a_{2}}}x_{3}\cdots x_{n-4}$$^{\underrightarrow{a}}n$ $^{\underleftarrow{c}}x_{n-3}$$^{\underrightarrow{a_{n-3}}}x_{n-1}$
 $^{\underleftarrow{a_{n-2}}}x_{n-2}$
  is an antidirected hamiltonian path in $H$; if $a_{n-3}\in A_{H}(x_{n-3},x_{n-1})$ and $n\leq 6$, then $n=5$ as $n\geq5$ is odd. Hence
  $n$ $^{\underleftarrow{a}}x_{2}$ $^{\underrightarrow{a_{2}}}x_{4}$$^{\underleftarrow{a_{3}}}x_{3}$ $^{\underrightarrow{a_{1}}}1$
   is an antidirected hamiltonian path in $H$ for $a_{1}\in A_{H}(x_{3},1)$ or
   $1$ $^{\underrightarrow{a_{1}}}x_{3}$ $^{\underleftarrow{c}}n$ $^{\underrightarrow{a_{3}}}x_{4}$$^{\underleftarrow{a_{2}}}x_{2}$
  is an antidirected hamiltonian path in $H$ for $a_{1}\in A_{H}(1,x_{3})$.\\
 \textbf{Subcase 1.2} $c=b$.
\par Then arcs $a_{1},a_{2},\ldots,a_{n-2}$ all contain vertices 1 and $n$. If $c\in A_{H}(x_{n-2},x_{n-1})$, then $P'=
1$ $^{\underrightarrow{a_{1}}}x_{2}$ $^{\underleftarrow{a_{2}}}x_{3}\cdots x_{n-3}$
$^{\underleftarrow{a_{n-3}}}x_{n-2}$ $^{\underrightarrow{c}}x_{n-1}$ is an antidirected path in $H$. Since $a_{n-2}\not=b$, one can construct an antidirected hamiltonian path in $H$ from $P'$ as in Subcase 1.1 when $c$ is replaced by $a_{n-2}$. Hence from now on assume that
\begin{equation*}
  \mbox{$c\in A_{H}(x_{n-1},x_{n-2})$.}
\end{equation*}
\par If $c\in A_{H}(1,x_{n-1})$, then
$x_{n-1}$$^{\underleftarrow{c}}1$ $^{\underrightarrow{a_{1}}}x_{2}$ $^{\underleftarrow{a_{2}}}x_{3}\cdots x_{n-3}$$^{\underleftarrow{a_{n-3}}}x_{n-2}$ $^{\underrightarrow{a}}n$
 is an antidirected hamiltonian path in $H$; if $c\in A_{H}(x_{n-1},1)$ and $a_{n-2}\in A_{H}(x_{n-2},1)$, then
$x_{2}$ $^{\underleftarrow{a_{2}}}x_{3}$ $^{\underrightarrow{a_{3}}}x_{4}\cdots x_{n-2}$$^{\underrightarrow{a_{n-2}}}1$ $^{\underleftarrow{c}}x_{n-1}$$^{\underrightarrow{a}}n$
  is an antidirected hamiltonian path in $H$; if $c\in A_{H}(x_{n-1},1)$ and $a_{n-2}\in A_{H}(1,x_{n-2})$, then
 $n$$^{\underleftarrow{a}}$$x_{n-1}$ $^{\underrightarrow{c}}x_{n-2}$$^{\underleftarrow{a_{n-2}}}1$ $^{\underrightarrow{a_{1}}}x_{2}$ $^{\underleftarrow{a_{2}}}x_{3}\cdots x_{n-4}$$^{\underrightarrow{a_{n-3}}}x_{n-3}$
 is an antidirected hamiltonian path in $H$.\\
\textbf{Case 2 }$k<n-1$.
\par  If there exists a vertex, denoted as $n$, such that a new $k$-hypertournament $H'$ is isomorphic to a member of $\{T^{c}_{3},T^{c}_{5},T^{c}_{7}\}$, obtained from $H$ by deleting the vertex $n$ along with all arcs in $A(H)$ containing the vertex $n$, then $k=2$ and assume that vertex set and arc set of $H'$ are depicted as in Figure 1. Since for any vertex $z_{i}\in V(H')$, $(z_{i},n)\in A(H)$ or $(n,z_{i})\in A(H)$.
Without loss of generality, we may assume that $(z_{1},n),(z_{2},n)\in A(H)$ (because $T^{c}_{3}\cong \overleftarrow{T}^{c}_{3},T^{c}_{5}\cong \overleftarrow{T}^{c}_{5} $ and $T^{c}_{7}\cong \overleftarrow{T}^{c}_{7}$). Thus,
$z_{3}\leftarrow z_{1}\rightarrow n\leftarrow z_{2}$
  is an antidirected hamiltonian path in $H'$ for $H'\cong T^{c}_{3}$ or $z_{1}\rightarrow n\leftarrow z_{2}\rightarrow z_{4}\leftarrow z_{3}\rightarrow z_{5}$
  is an antidirected hamiltonian path in $H'$ for $H'\cong T^{c}_{5}$ or $z_{1}\rightarrow n\leftarrow z_{2}\rightarrow z_{6}\leftarrow z_{5}\rightarrow z_{7} \leftarrow z_{3}\rightarrow z_{4}$
  is an antidirected hamiltonian path in $H'$ for $H'\cong T^{c}_{7}$.

\par
If there exists a vertex, denoted as $n$, such that a new $k$-hypertournament $H'$ is isomorphic to $H_{4}$, obtained from $H$ by deleting the vertex $n$ along with all arcs in $A(H)$ containing the vertex $n$, then $n=5$ and $k=3$. Let $A(H')=\{a_{1}=(2,3,4),a_{2}=(4,1,2), a_{3}=(3,4,1),a_{4}=(1,2,3)\}$. Clearly, $A(H')\subset A(H)$ and for any $i\in \{1,2,3,4\}$, $|A_{H}(i,5)\cup A_{H}(5,i)|=3$. Without loss of generality, assume that $|A_{H}(1,5)|\geq 2$. If $|A_{H}(2,5)|\geq 1$, then there exist two distinct arcs $a_{(2,5)}\in A_{H}(2,5)$ and $a_{(1,5)}\in A_{H}(1,5)$. Thus, $4$$^{\underleftarrow{a_{1}}}2$ $^{\underrightarrow{a_{(2,5)}}}5$ $^{\underleftarrow{a_{(1,5)}}}1 $$^{\underrightarrow{a_{4}}}3$
  is an antidirected hamiltonian path in $H$; if $|A_{H}(3,5)|\geq 1$, then there exist two distinct arcs $a_{(3,5)}\in A_{H}(3,5)$ and $a_{(1,5)}\in A_{H}(1,5)$. Thus, $4$$^{\underleftarrow{a_{1}}}3$ $^{\underrightarrow{a_{(3,5)}}}5$ $^{\underleftarrow{a_{(1,5)}}}1 $$^{\underrightarrow{a_{4}}}2$
  is an antidirected hamiltonian path in $H$. Hence assume that $|A_{H}(2,5)|=0$ and $|A_{H}(3,5)|=0$, then $|A_{H}(5,2)|=3$ and $|A_{H}(5,3)|=3$, and so there exist two distinct arcs $a_{(5,2)}\in A_{H}(5,2)$ and $a_{(5,3)}\in A_{H}(5,3)$. Thus, $4$$^{\underrightarrow{a_{2}}}2$ $^{\underleftarrow{a_{(5,2)}}}5$ $^{\underrightarrow{a_{(5,3)}}}3 $$^{\underleftarrow{a_{4}}}1$
  is an antidirected hamiltonian path in $H$.
\par Hence from now on assume that for any vertex $i$,
\begin{equation}\label{nh4}
  \mbox{the new $k$-hypertournament $H'$ is not isomorphic to a member of $\{T^{c}_{3},T^{c}_{5},T^{c}_{7},H_{4}\}$,}
\end{equation}
obtained from $H$ by deleting the vertex $i$ along with all arcs in $A(H)$ containing the vertex $i$.
\par We proceed by induction on $n \geq 4$. The case $n=4$ (and hence, $k=2$ ) follows from Theorem \ref{ADHT}. Therefore, suppose that $n \geq 5$. Consider the new $k$-hypertournament $H^{\prime \prime}$ obtained from $H$ by deleting the vertex $n$ along with all arcs in $A(H)$ containing the vertex $n$.
\par If $n=k+2$, then $H''$ is an $(n-2)$-hypertournament on $n-1$ vertices. By  $(\ref{nh4})$,  $H''\not\in \{H_{4},T^{c}_{3},T^{c}_{5},T^{c}_{7}\}$.
Hence, $H''$ has an antidirected hamiltonian path because of Case 1.

If $n> k+2$, then by $(\ref{nh4})$, $H''\not\in \{H_{4},T^{c}_{3},T^{c}_{5},T^{c}_{7}\}$ and by the induction hypothesis, $H''$ has an antidirected hamiltonian path.
\par In conclusion, $H''$ has an antidirected hamiltonian path. Assume that $P=x_{1}$ $^{\underrightarrow{a_{1}}}x_{2}$ $^{\underleftarrow{a_{2}}}x_{3}\cdots x_{n-1}$ is an antidirected hamiltonian path in $H''$ such that $x_{1}$ is a starting vertex. Then $P$ is an antidirected path of $H$. Assume first that $a_{n-2}\in A_{H}(x_{n-2},x_{n-1})$, then $n$ is odd.\\
\textbf{Subcase 2.1} $k=3$.
\par Assume first that there exists an integer $i_{1}\in \{1,3,\ldots,n-2\}$ such that
$A_{H}(x_{i_{1}},n)\not=\emptyset, A_{H}(n, x_{i_{1}+1})$
$=\emptyset,A_{H}(
 x_{i_{1}+2},n)=\emptyset,\ldots, A_{H}(
 x_{n-2},n)=\emptyset,A_{H}(n,
 x_{n-1})=\emptyset$, then $|A_{H}(x_{i_{1}+1},$
 $x_{i_{1}+2})|\geq |A_{H}(x_{i_{1}+1},n)\cap A_{H}(n,x_{i_{1}+2})|$
 $=1, \ldots,|A_{H}(x_{n-1},$
 $x_{n-2})|\geq |A_{H}(x_{n-1},n)\cap A_{H}(n,x_{n-2})|$
 $=1$.
As $k=3$, we have $ A_{H}(x_{i_{1}},n),$
$ A_{H}(x_{i_{1}+1},n)\cap A_{H}(n,x_{i_{1}+2}),\ldots,$
$ A_{H}(x_{n-1},n)\cap A_{H}(n,x_{n-2})$ are mutually arc-disjoint and
$ A_{H}(x_{i_{1}+1},n), A_{H}(x_{i_{1}+3},n)\cap A_{H}(n,x_{i_{1}+2}), $
$\ldots, A_{H}(x_{n-1},n)\cap A_{H}(n,$
$x_{n-2})$ are mutually arc-disjoint. Since $A_{H}(n, x_{i_{1}+1})=\emptyset$, we have that $|A_{H}( x_{i_{1}+1},n)|= \binom{n-2}{k-2}=\binom{n-2}{1}=n-2$,
and then
  there exist $n-i_{1}$ distinct arcs $a_{(x_{i_{1}},n)}\in A_{H}(x_{i_{1}},n), a_{(x_{i_{1}+1},x_{i_{1}+2})}$
  $\in A_{H}(x_{i_{1}+1},$
  $x_{i_{1}+2}),\ldots, a_{(x_{n-1},x_{n-2})}\in A_{H}(x_{n-1},x_{n-2})$ and $a_{(x_{i_{1}+1},n)}\in A_{H}$
  $(x_{i_{1}+1},n)$.
 Thus,
$x_{1}$ $^{\underrightarrow{a_{1}}}x_{2}$ $^{\underleftarrow{a_{2}}}x_{3}\cdots x_{i_{1}}$
$^{\underrightarrow{a_{(x_{i_{1}},n)}}}n$
$^{\underleftarrow{a_{(x_{i_{1}+1},n)}}}x_{i_{1}+1}$
$^{\underrightarrow{a_{(x_{i_{1}+1},x_{i_{1}+2})}}}x_{i_{1}+2}\cdots x_{n-2}$
$^{\underleftarrow{a_{(x_{n-1},x_{n-2})}}}x_{n-1}$
 is an antidirected hamiltonian path in $H$.
\par Likewise, if there exists an integer $i_{2}\in \{2,4,\ldots,n-1\}$ such that
$A_{H}(n ,x_{i_{2}})\not=\emptyset, A_{H}( x_{i_{2}+1},n)$
$=\emptyset,A_{H}(n,
 x_{i_{2}+2})=\emptyset\ldots, A_{H}( x_{n-2},n)=\emptyset,A_{H}(n,
 x_{n-1})=\emptyset$, then $H$ contains an antidirected hamiltonian path in $H$.
\par Assume now that for any integer $i\in \{1,2,\ldots,n-1\}$, if $i$ is odd, then $A_{H}(x_{i},n)=\emptyset$; if $i$ is even, then $A_{H}(n,x_{i})=\emptyset$.
Thus, $|A_{H}(x_{2},x_{1})|\geq |A_{H}(x_{2},n)\cap A_{H}(n,x_{1})|=1, \ldots,|A_{H}(x_{n-1},x_{n-2})|$
$\geq |A_{H}(x_{n-1},n)\cap A_{H}(n,x_{n-2})|=1$. As $k=3$, we have that $ A_{H}(x_{2},n)\cap A_{H}(n,x_{1}),\ldots, A_{H}(x_{n-1},n)$
$\cap A_{H}(n,x_{n-2})$ are mutually arc-disjoint and $A_{H}(n,x_{1}), A_{H}(x_{2},n)\cap A_{H}(n,x_{3}),\ldots,A_{H}(x_{n-1},n)\cap A_{H}(n,x_{n-2})$ are mutually arc-disjoint. Since $A_{H}(x_{1},n)=\emptyset$, we have that $|A_{H}( n,x_{1})|= n-2$, and then there exist $n-1$ distinct arcs $a_{(x_{2},x_{1})}\in A_{H}(x_{2},x_{1}),\ldots, a_{(x_{n-1},x_{n-2})}$
$\in A_{H}(x_{n-1},x_{n-2})$ and $a_{(n,x_{1})}\in A_{H}(n,x_{1})$.
 Thus,
$n$$^{\underrightarrow{a_{(n,x_{1})}}}x_{1}$ $^{\underleftarrow{a_{(x_{2},x_{1})}}}x_{2}$ $^{\underrightarrow{a_{(x_{2},x_{3})}}}x_{3}\cdots $
$ x_{n-2}$$^{\underleftarrow{a_{(x_{n-1},x_{n-2})}}}$
$x_{n-1}$
 is an antidirected hamiltonian path in $H$.\\
 \textbf{Subcase 2.2} $k=4$.
\par Assume first that there exists an integer $i_{1}\in \{1,3,\ldots,n-2\}$ such that
$A_{H}(x_{i_{1}},n)\not=\emptyset, A_{H}(n, x_{i_{1}+1})$
$=\emptyset,A_{H}(
 x_{i_{1}+2},n)=\emptyset,\ldots, A_{H}(
 x_{n-2},n)=\emptyset,A_{H}(n,
 x_{n-1})=\emptyset$, then $|A_{H}(x_{i_{1}+1},$
 $x_{i_{1}+2})|\geq |A_{H}(x_{i_{1}+1},n)\cap A_{H}(n,x_{i_{1}+2})|$
 $= n-3, \ldots,$
 $|A_{H}(x_{n-1},x_{n-2})|$
 $\geq |A_{H}(x_{n-1},n)\cap A_{H}(n,$
 $x_{n-2})|= n-3$.
As $k=4$, we have that $(A_{H}(x_{n-1},n)\cap A_{H}(n,x_{n-2}))\cap (A_{H}(x_{n-3},n)\cap A_{H}(n,x_{n-4}))=\emptyset, \ldots, (A_{H}(x_{n-1},n)\cap A_{H}(n,x_{n-2}))$
$\cap (A_{H}(x_{i_{1}+1},n)\cap A_{H}(n,x_{i_{1}+2}))=\emptyset$ and $|A_{H}( x_{i_{1}+1},n)|= \binom{n-2}{k-2}\geq\binom{n-2}{2}\geq n$, and then there exist $n-i_{1}$ distinct arcs $a_{(x_{i_{1}},n)}\in A_{H}(x_{i_{1}},n), a_{(x_{i_{1}+1},x_{i_{1}+2})}$
  $\in A_{H}(x_{i_{1}+1},$
  $x_{i_{1}+2}),$
  $\ldots, a_{(x_{n-1},x_{n-2})}\in A_{H}(x_{n-1},x_{n-2})$ and $a_{(x_{i_{1}+1},n)}\in
  A_{H}$
  $(x_{i_{1}+1},n)$.
 Thus,
$x_{1}$ $^{\underrightarrow{a_{1}}}x_{2}$ $^{\underleftarrow{a_{2}}}x_{3}\cdots x_{i_{1}}$$^{\underrightarrow{a_{(x_{i_{1}},n)}}}n$
$^{\underleftarrow{a_{(x_{i_{1}+1},n)}}}$
$x_{i_{1}+1}$
$^{\underrightarrow{a_{(x_{i_{1}+1},x_{i_{1}+2})}}}x_{i_{1}+2}\cdots x_{n-2}$
$^{\underleftarrow{a_{(x_{n-1},x_{n-2})}}}$
$x_{n-1}$
 is an antidirected hamiltonian path in $H$.
\par Likewise, if there exists an integer $i_{2}\in \{2,4,\ldots,n-1\}$ such that
$A_{H}(n ,x_{i_{2}})\not=\emptyset, A_{H}( x_{i_{2}+1},n)$
$=\emptyset,A_{H}(n,
 x_{i_{2}+2})=\emptyset\ldots, A_{H}( x_{n-2},n)=\emptyset,A_{H}(n,
 x_{n-1})=\emptyset$, then $H$ contains an antidirected hamiltonian path in $H$.
\par Assume now that for any integer $i\in \{1,2,\ldots,n-1\}$, if $i$ is odd, then $A_{H}(x_{i},n)=\emptyset$; if $i$ is even, then $A_{H}(n,x_{i})=\emptyset$.
Thus, $|A_{H}(x_{2},x_{1})|\geq |A_{H}(x_{2},n)\cap A_{H}(n,x_{1})|=n-3, \ldots,|A_{H}(x_{n-1},x_{n-2})|\geq |A_{H}(x_{n-1},n)\cap A_{H}(n,x_{n-2})|=n-3$. Similarly, there exist $n-1$ distinct arcs $a_{(x_{2},x_{1})}\in A_{H}(x_{2},x_{1}),\ldots,$
$ a_{(x_{n-1},x_{n-2})}\in A_{H}(x_{n-1},x_{n-2})$ and $a_{(n,x_{1})}\in A_{H}(n,x_{1})$.
 Thus,
$n$$^{\underrightarrow{a_{(n,x_{1})}}}x_{1}$ $^{\underleftarrow{a_{(x_{2},x_{1})}}}x_{2}$ $^{\underrightarrow{a_{(x_{2},x_{3})}}}x_{3}\cdots x_{n-2}$
$^{\underleftarrow{a_{(x_{n-1},x_{n-2})}}}x_{n-1}$
 is an antidirected hamiltonian path in $H$.\\
 \textbf{Subcase 2.3} $5\leq k\leq n-2$.
\par Assume first that there exists an integer $i_{1}\in \{1,3,\ldots,n-2\}$ such that
$A_{H}(x_{i_{1}},n)\not=\emptyset, A_{H}(n, x_{i_{1}+1})$
$=\emptyset,A_{H}(
 x_{i_{1}+2},n)=\emptyset,\ldots, A_{H}(
 x_{n-2},n)=\emptyset,A_{H}(n,
 x_{n-1})=\emptyset$, then $|A_{H}(x_{i_{1}+1},$
 $x_{i_{1}+2})|\geq |A_{H}(x_{i_{1}+1},n)\cap A_{H}(n,x_{i_{1}+2})|$
 $\geq n-1, \ldots,$
 $|A_{H}(x_{n-1},x_{n-2})|$
 $\geq |A_{H}(x_{n-1},n)\cap A_{H}(n,x_{n-2})|$
 $\geq n-1$ and $|A_{H}
  (x_{i_{1}+1},n)|> n$, and so there exist $n-i_{1}$ distinct arcs $a_{(x_{i_{1}},n)}\in A_{H}(x_{i_{1}},n), $
  $a_{(x_{i_{1}+1},x_{i_{1}+2})}$
  $\in A_{H}(x_{i_{1}+1},x_{i_{1}+2}),$
  $\ldots, a_{(x_{n-1},x_{n-2})}\in A_{H}(x_{n-1},x_{n-2})$ and $a_{(x_{i_{1}+1},n)}\in A_{H}$
  $(x_{i_{1}+1},n)$.
 Thus,
$x_{1}$ $^{\underrightarrow{a_{1}}}x_{2}$ $^{\underleftarrow{a_{2}}}x_{3}\cdots x_{i_{1}}$
$^{\underrightarrow{a_{(x_{i_{1}},n)}}}n$
$^{\underleftarrow{a_{(x_{i_{1}+1},n)}}}x_{i_{1}+1}$
$^{\underrightarrow{a_{(x_{i_{1}+1},x_{i_{1}+2})}}}x_{i_{1}+2}\cdots x_{n-2}$
$^{\underleftarrow{a_{(x_{n-1},x_{n-2})}}}x_{n-1}$
 is an antidirected hamiltonian path in $H$.
\par Likewise, if there exists an integer $i_{2}\in \{2,4,\ldots,n-1\}$ such that
$A_{H}(n ,x_{i_{2}})\not=\emptyset, A_{H}( x_{i_{2}+1},n)$
$=\emptyset,A_{H}(n,
 x_{i_{2}+2})=\emptyset\ldots, A_{H}( x_{n-2},n)=\emptyset,A_{H}(n,
 x_{n-1})=\emptyset$, then $H$ contains an antidirected hamiltonian path in $H$.
\par Assume now that for any integer $i\in \{1,2,\ldots,n-1\}$, if $i$ is odd, then $A_{H}(x_{i},n)=\emptyset$; if $i$ is even, then $A_{H}(n,x_{i})=\emptyset$.
Thus, $|A_{H}(x_{2},x_{1})|\geq |A_{H}(x_{2},n)\cap A_{H}(n,x_{1})|\geq n-1, \ldots,|A_{H}(x_{n-1},x_{n-2})|\geq |A_{H}(x_{n-1},n)\cap A_{H}(n,x_{n-2})|\geq n-1$ and $|A_{H}(n,x_{1})|> n$, and so there exist $n-1$ distinct arcs $a_{(x_{2},x_{1})}\in A_{H}(x_{2},x_{1}),\ldots,$
$ a_{(x_{n-1},x_{n-2})}\in A_{H}(x_{n-1},x_{n-2})$ and $a_{(n,x_{1})}\in A_{H}(n,x_{1})$.
 Thus,
$n$$^{\underrightarrow{a_{(n,x_{1})}}}x_{1}$ $^{\underleftarrow{a_{(x_{2},x_{1})}}}x_{2}$ $^{\underrightarrow{a_{(x_{2},x_{3})}}}x_{3}\cdots x_{n-2}$
$^{\underleftarrow{a_{(x_{n-1},x_{n-2})}}}x_{n-1}$
 is an antidirected hamiltonian path in $H$.
 \par Therefore, if $a_{n-2}\in A_{H}(x_{n-2},x_{n-1})$, then $H$ has an antidirected hamiltonian path.
Likewise, if $a_{n-2}\in A_{H}(x_{n-1},x_{n-2})$ or $H''$ contains an antidirected hamiltonian path such that two ending vertices are terminal vertices, then $H$ has an antidirected hamiltonian path. This completes the proof of the theorem.
\end{proof}

\begin{flushleft}
\textbf{Declarations}
\end{flushleft}
\par The authors declare that they have no known competing financial interests or personal relationships that could have
appeared to influence the work reported in this paper and no data was used for the research described in the paper.

\end{document}